\documentclass{amsart} \usepackage{amssymb,amsfonts,amscd}
\newtheorem{theorem}{Theorem}[section]
\newtheorem{lemma}[theorem]{Lemma}
\newtheorem{corollary}[theorem]{Corollary}
\newtheorem{proposition}[theorem]{Proposition}
\theoremstyle{remark}

\theoremstyle{definition}

\numberwithin{equation}{section} \makeatother

\DeclareMathOperator{\Cdb}{{\mathbb C}}

\DeclareMathOperator{\Ndb}{{\mathbb N}}

\begin{document}

\title[Open projections in operator algebras II]{Open projections in operator algebras II: Compact projections} \author[David P. Blecher]{David P. Blecher} \author[Matthew Neal]{Matthew Neal}

\address{Department of Mathematics, University of Houston, Houston, TX
77204-3008}
 \email[David P.
Blecher]{dblecher@math.uh.edu}
\address{Department of Mathematics,
Denison University, Granville, OH 43023}
\email{nealm@denison.edu}
\thanks{The first author was supported by an extension of grant
0800674 from the National Science Foundation.
The second author was supported by Denison University.  Revision of March 15, 2012.
To appear, Studia Mathematica} \begin{abstract}
We  generalize some aspects of the theory of
compact projections relative to a $C^*$-algebra, to the setting of
more general algebras.
Our main result is that compact projections are the decreasing limits
of `peak projections', and in the separable case compact projections  are just the
peak projections.  We also establish new forms of the noncommutative Urysohn lemma
relative to an operator algebra, and we show that a projection is compact
iff the associated face in the state space of the algebra is weak* closed.
  \end{abstract}



\maketitle

\let\text=\mbox

\section{Introduction and notation}

For us, an operator algebra is a closed algebra of operators on 
a Hilbert space.    Selfadjoint operator algebras, or $C^*$-algebras,
are often thought of as a kind of noncommutative topology, and 
one explicit and important manifestation of this is 
Akemann's noncommutative topology.
He introduced (see e.g.\ \cite{Ake,Ake2,AAP})
 the open, closed, and compact projections,
certain projections in the second dual of the algebra, which
generalize open, closed, and compact subsets of a 
topological space.    
In \cite{BHN,Hay} the authors and Hay generalized much of the 
 theory of open and closed
projections in $C^*$-algebras, to 
 general operator algebras, thereby initiating a   
`noncommutative topology relative to a
general operator algebra'.  This should 
be useful in `noncommutative function theory' in the 
ways that {\em peak sets} and related 
tools have been useful in the study of function
spaces (see e.g.\ \cite{Gamelin}).   This study 
was considerably advanced in \cite{BRead}.  In \cite{BN0} the authors 
studied the generalization of 
 compact projections to the setting of ternary 
rings of operators (TROs) (see also e.g.\ \cite{AP,ER4,FP0}).  In the present paper we
begin to study the analogue of compact projections in the 
setting of 
general operator algebras.   We do this for the primary reason that
this topic is important in its own right, and should have 
many future applications in `noncommutative function theory'.
To a lesser extent, we hope that our results will eventually 
find some application in the comparison theory program initiated 
in our recent paper \cite{BN}.   Indeed in the $C^*$-algebra
case (see e.g.\ \cite{ORT}),  the important topic of 
Cuntz equivalence and subequivalence is intimately connected to
 compact projections.     One should beware though,
direct attempts to apply our results to Cuntz comparison face significant problems.
  Indeed simple examples
such as variants of the disk algebra, seem to indicate that for Cuntz comparison one 
would probably mostly need to use compact
projections with respect to a containing $C^*$-algebra, rather than 
with respect to the algebra itself.

The study of compact projections in the second dual of an 
operator algebra $A$ turns out to be closely related to
the topic of `Urysohn type lemmas' for $A$.  The noncommutative Urysohn lemma
for $C^*$-algebras was proved by Akemann \cite{Ake}, with later refinements
by him and coauthors (see our bibliography), and by L. G. Brown \cite{Brown}.
We begin Section 
2 below with a noncommutative Urysohn lemma for nonunital algebras, which involves
an $\epsilon$ that does not appear in the $C^*$-algebraic Urysohn lemma.
One may not remove this $\epsilon$ in general for nonselfadjoint algebras. 
 However if the two projections
involved lie in the second dual of $A$, then the situation 
is much better: one may remove the $\epsilon$ and obtain a much stronger 
noncommutative Urysohn lemma.   
We define two variants of the notion of a 
compact projection which we show are equivalent, one in terms of the set $\{ a \in A : \Vert 1 - 2 a \Vert 
\leq 1 \}$.  The latter set, which is written as $\frac{1}{2} \mathfrak{F}_A$,  
was shown in \cite{BRead,BN} to often play the role of the 
positive cone in a $C^*$-algebra, and so 
we imagine that in future it will sometimes be important to 
have the $\frac{1}{2} \mathfrak{F}_A$ formulation.    
In Section 3 we  consider the nonunital case of Hay's {\em peak projections}.
Our main result is that compact projections are the decreasing limits
of  
`peak projections', and in the separable case are just the peak projections.
(There is a point of possible confusion here:
if $A$ is nonunital then our main  result is quite different from 
\cite[Corollary 2.21]{BRead}, although the latter is one of
the ingredients of our result.  This is because our `peak projections' are 
different (see the last remark in Section 3 below).   In the unital 
case the results essentially coincide, but in the present paper
the nonunital case is the only case
that is new and of interest (since 
in the unital case compactness is the same as being closed).
One may of course use \cite[Corollary 2.21]{BRead} to get a similar sounding
statement in terms of the unitization $A^1$, but we do not see
how to connect this statement easily to our main result here.)   
  We  prove in Section 4 that a projection in $A^{**}$ is compact
iff the associated face in the state space $S(A)$ of the algebra is weak* closed.   
Thus compact projections correspond to certain weak* closed faces of 
 $S(A)$.  In Section 5 we make some remarks on pure states and minimal 
projections.

As we have said, in our paper we are generalizing, and are inspired by,
 results from $C^*$-algebra theory and their proofs.  However 
for more general algebras there are significant obstacles
to be overcome, some of which use deep results from earlier papers
of ours and our coauthors.  
Although we use nothing essential from JB$^*$-triple theory in the present paper,
we mention that many results in our paper have matching counterparts
with quite different proofs in that theory (see our bibliography 
below).  In particular,
\cite{FP} contains  JB$^*$-triple Urysohn lemmas (see
also e.g.\ \cite{FP1,BN0}), which are considerably deeper than the $C^*$-algebra
variants.
 
Throughout this paper, $A$ is a fixed operator algebra, and $B$ is
a $C^*$-algebra that contains it.  We will also assume throughout that 
$A$ is {\em approximately unital},
that is it has a
contractive approximate
identity (cai), although this probably is not necessary for many of the results.
As we said above we are usually {\em not} interested in the case that $A$ is unital.
  We will  assume sometimes 
(in Sections 4 and 5 particularly) that the cai for $A$ is also a cai for $B$;
this is automatic by \cite[Lemma 2.1.6]{BLM} if $B$ is generated by $A$ 
(that is there is no proper $C^*$-subalgebra of $B$ containing $A$).
  
We now discuss notation and background, for which it will be helpful
for the reader to have easy access to several of the references, particularly 
\cite{BLM,BHN,BRead}, for more details beyond what is presented here.  
For us a {\em projection}
is always an orthogonal projection.
We recall that by a theorem due to  Ralf Meyer,
every operator algebra $A$ has a unique unitization $A^1$ (see
e.g.\ \cite[Section 2.1]{BLM}). Below $1$ always refers to
the identity of $A^1$ if $A$ has no identity. 
If $A$ is a nonunital
operator algebra represented (completely) isometrically on a Hilbert
space $H$ then one may identify the unitization $A^1$ with the 
algebra $A + \Cdb I_H$.   The
second dual $A^{**}$ is also an operator algebra with its (unique)
Arens product, this is the product inherited from the von Neumann
algebra $B^{**}$ if
$A$ is a subalgebra of a $C^*$-algebra $B$.  Meets and joins in
$B^{**}$ of projections in $A^{**}$ remain in $A^{**}$,
since these meets and joins may be computed in the biggest
von Neumann algebra contained inside $A^{**}$. Note that
$A$ has a cai iff $A^{**}$ has an identity $1_{A^{**}}$ of norm $1$,
and then $A^1$ is sometimes identified with $A + \Cdb 1_{A^{**}}$.
If $A$ has a cai, then a {\em state} of $A$ is a functional $\varphi
\in {\rm Ball}(A^*)$ with $\varphi(e_t) \to 1$, for some (or every)
cai $(e_t)$ for $A$.  We write $S(A)$ for the space of states.  States extend uniquely to states on the unitization $A^1$
(see \cite[2.1.19]{BLM}).
If $B$ is a  $C^*$-algebra
generated by $A$,
 then it is known that any bounded approximate identity
(bai) for $A$ is a bai for
$C^*(A)$, and hence states of $A$ are precisely the restrictions to
$A$ of states on $C^*(A)$ (see \cite[2.1.19]{BLM}).   It follows
that  the {\em quasistate} space
 $Q(A)$ is weak* compact, where $Q(A) = \{ t \varphi : t \in [0,1],
\varphi \in S(A) \}$.  Indeed if $\varphi_t \in S(B), c_t \in [0,1],
\varphi \in A^{**},$
and $c_t \varphi_t \to \varphi$ weak* on $A$, then there is a $\psi \in Q(B)$
and  convergent subnets
$c_{t_\mu} \to c \in [0,1], \varphi_{t_\mu} \to \psi$ weak* on $B$.  This forces
$\varphi = c \psi \in Q(A)$.  So $Q(A)$ is weak* compact.    

A {\em hereditary subalgebra} (HSA) of $A$ is an
subalgebra $D$ which has a cai and satisfies $DAD \subset D$.   For the theory of HSA's in general operator algebras
see \cite{BHN}.  These objects are in an order preserving,
bijective correspondence with the {\em open projections} $p \in A^{**}$, by which we mean that there
is a net $x_t \in A$ with $x_t = p x_t p \to p$ weak*.  These are
also the open projections $p$ in the sense of Akemann 
\cite{Ake,Ake2} in $B^{**}$, where $B$ is a $C^*$-algebra containing $A$, such that
$p \in A^{\perp \perp}$.  Indeed the weak* limit of a cai for a HSA is
an open projection, and is called the {\em support projection} of the HSA.
Conversely, if $p$ is an open projection in $A^{**}$, then
 $pA^{**}p \cap A$ is a  HSA in $A$.  If an approximately unital
operator algebra $A$ is viewed as a HSA in its unitization, then
we will write its support projection as $e$, or as $1$ if there is no danger of confusion.   

We recall that a {\em closed projection} is the `perp' of 
an open projection.   Suprema (resp.\ infima) in $B^{**}$ of open
(resp.\ closed) projections  in $A^{**}$, remain in $A^{**}$,
by the fact mentioned two paragraphs earlier about meets and joins,
together with the $C^*$-algebraic case of these facts \cite{Ake,Ake2}.

We will use the notation $\mathfrak{F}_A$ for 
$\{ a \in A  :
\Vert 1 - a \Vert \leq 1 \}$, and 
$\frac{1}{2} \mathfrak{F}_A$ for $\{ a \in A :
\Vert 1 - 2a \Vert \leq 1 \}$, a subset of Ball$(A)$.
In \cite{BRead} it is proved that elements in
$\mathfrak{F}_A$ (resp.\ $\frac{1}{2} \mathfrak{F}_A$)
have $n$th roots for all $n \in \Ndb$,
which are again in $\mathfrak{F}_A$ (resp.\ $\frac{1}{2} \mathfrak{F}_A$).  If $a \in 
\mathfrak{F}_A$, then $(a^{\frac{1}{n}})$ converges
weak* to an open projection which is written as
 $s(a)$, and this is both the left
and the right support projection of $a$ (see \cite[Section 2]{BRead}).
In this case $\overline{aAa}$ is a HSA of $A$, and the
 support projection of this HSA is $s(a)$,
 so $\overline{aAa} = s(a) A^{**} s(a)
 \cap A$.

We recall that a {\em tripotent} is an element $u$ with
$u u^* u = u$.   We order tripotents by $u \leq v$ iff
$u u^* v = u$.    If $x  \in {\rm Ball}(B)$,  define $u(x)$ to
be the weak* limit of the sequence $(x (x^{*}x)^n)$ in $B^{**}$.
This is the largest tripotent in $B^{**}$ satisfying
$v v^* x = v$ (see \cite{ER4}).  It is well known that if $\psi \in {\rm Ball}(A^*)$, then
$\psi(x) = 1$  iff $\psi(u(x))  = 1$ (see
e.g.\ \cite[Lemma 3.3 (i)]{ER4}).  In fact this is
easy to prove  in a few lines using spectral theory.  
The following result, essentially due to Edwards and R\"uttimann,
 has been used elsewhere in our work, and is no doubt 
very well known (cf.\ e.g.\ \cite[Proposition 5.5]{Hay}).

\begin{proposition} \label{uconv} Suppose $x$ and $y$ lie in the unit ball of a
 $C^{*}$-algebra $B$. Then $u(\frac{x+y}{2})=u(x)\wedge u(y)$ in $B^{**}$
(that is, $u(\frac{x+y}{2})$ is the largest tripotent in $B^{**}$ dominated by
both $u(x)$ and $u(y)$ in the ordering of tripotents above).
\end{proposition}
\begin{proof}
In \cite[Theorem 4.4]{ERF} it is proved that
there is an order isomorphism from the
set of tripotents in $W = B^{**}$,
onto a certain set
 of closed faces of ${\rm Ball}(B^*)$.   This mapping
takes a tripotent $u$ to the face
$\{ u \}_{\prime}$ defined by $\{ \varphi \in {\rm Ball}(B^*) \, : \,   \varphi(u) = 1 \}$.

By the remark above the proposition,
 $\{u(\frac{x+y}{2})\}_{\prime}$
is the set of $\psi \in {\rm Ball}(A^*)$  with
$\psi(\frac{x+y}{2}) = 1$.  But this clearly happens iff
$\psi(x) = \psi(y) = 1$; that is iff $\psi \in \{u(x)\}_{\prime} \cap  \{u(y)\}_{\prime}$.The latter equals $\{u(x) \wedge u(y)\}_{\prime}$,
because of the order isomorphism mentioned in the last paragraph.
Hence $\{u(\frac{x+y}{2})\}_{\prime} = \{u(x) \wedge u(y)\}_{\prime}$,
and so $u(\frac{x+y}{2})=u(x)\wedge u(y)$ because of the 
isomorphism in the last paragraph again.
\end{proof}

\section{Compact projections and the Urysohn lemma} 

We recall again that  throughout  $A$ is an
 operator algebra which is 
  approximately unital, although as we said
this probably is not necessary for many of the results.
Also, $B$ is a $C^*$-algebra containing $A$.  
We will say that 
a closed projection $q \in A^{**}$ is {\em compact} in $A^{**}$ if there exists 
$a \in {\rm Ball}(A)$  with $q = qa$.    We say that such $q$ is ${\mathfrak F}$-{\em compact} in $A^{**}$ if the $a$ here may be chosen in $\frac{1}{2} {\mathfrak F}_A$.
We will prove below that  every compact projection is ${\mathfrak F}$-compact.
 
Any  compact projection $q$ in  $A^{**}$  is a 
compact projection in $B^{**}$, since in this case
it is easy to argue from elementary operator 
theory that we have $q = q a^* = q a a^*$, for $a$ as above.  Clearly 
any closed projection is compact in $A^{**}$ if $A$ is unital.  Any
closed projection dominated by a compact projection in $A^{**}$ is 
compact.    If $q$ is a  projection in $A^{**}$, and $q b = q$
for some $b \in  {\rm Ball}(A)$, then
$q u(b) = q$, and $q \leq u(b)$.    Here  $u(b)$ is the tripotent mentioned in the 
introduction, namely  the weak* limit of $b (b^* b)^n$.  However $u(b)$ may not be a
projection; soon we will be able to rechoose $b$ so that it is.
  If $q$ is a compact projection,
and if $(e_t)$ is a cai for the HSA supported by $e-q$ (where $e = 1_{A^{**}}$),
then $e - e_t \to q$ weak*.  Let $y_t =  b
 (e - e_t)
\in A$.  Then $y_t \to b (e - (e-q))  = 
b q = q$.  
We have $y_t \in {\rm Ball}(A)$, and $y_t q = b (q - e_t q) 
= b q  = q$.

\begin{theorem} \label{peakch} {\rm (A noncommutative Urysohn lemma for
approximately unital operator algebras;
c.f.\ Theorem 2.24 in \cite{BRead}.) } \ Let $A$ be an
approximately unital operator algebra, a subalgebra of a $C^*$-algebra $B$, 
and let
$q$ be a compact projection  in $A^{**}$.  Then
 for any open
projection $u \in B^{**}$ with $u \geq q$, and any $\epsilon > 0$,
there exists an $a \in {\rm Ball}(A)$ with
$a q = q$ and $\Vert a (1-u) \Vert < \epsilon$
and $\Vert (1-u)  a \Vert < \epsilon$.
\end{theorem}

 \begin{proof}
Let $q  \in A^{\perp \perp}$, let $u$ be an open
projection with $u \geq q$, and let $\epsilon > 0$ be given.  As above,
 let $(y_t)$ be a net in ${\rm Ball}(A)$ with $y_t q = q$ and $y_t \to q$ weak*.   
We follow
the idea in the last seven lines
of the proof of \cite[Theorem 6.4]{BHN} (see also \cite[Theorem 2.24]{BRead}):
By the noncommutative Urysohn lemma
\cite{Ake}, there is an $x \in B$ with
$q \leq x \leq u$.
Then  $y_t (1-x) \to q (1-x) = 0$ weak*, and hence weakly
in $B$.  Similarly, $(1-x) y_t \to  0$ weakly.
By a routine convexity argument in $B \oplus B$, given $\epsilon > 0$
there is a convex combination $a$ of the $y_t$
such that $\Vert a (1-x) \Vert < \epsilon$ and
$\Vert  (1-x) a \Vert < \epsilon$.  Clearly
 $a q = q$.  Therefore $\Vert a (1-u) \Vert =
\Vert a (1-x) (1-u)  \Vert < \epsilon$.
Similarly for $\Vert   (1-u) a \Vert  < \epsilon$.
 \end{proof}

For a $C^*$-algebra one may remove the $\epsilon$ in the last theorem, but this
is not necessarily true for more general algebras.   
 However we will see later that 
 things are much better, and the  noncommutative Urysohn lemma can be refined,
 if we assume that 
the two projections $u$ and $q$ involved lie in $A^{**}$.
  
\begin{theorem} \label{perun} If $A$ is an approximately unital  
operator algebra,  a subalgebra of a $C^*$-algebra $B$.  If $q$ is 
a projection in $A^{**}$ then the following  are equivalent:
 \begin{enumerate} \item [{\rm (i)}]  $q$ is
compact in $B^{**}$,    \item [{\rm (ii)}]  $q$ is a closed projection  in $(A^1)^{**}$,
\item [{\rm (iii)}]  
$q$ is
compact in $A^{**}$,
  \item [{\rm (iv)}]  $q$ is
${\mathfrak F}$-compact in $A^{**}$. \end{enumerate}
\end{theorem}  \begin{proof}   We may assume that $A$ is nonunital, and that
$1_{A^1} = 1_{B^1}$.  

 {\rm (iv)} $\Rightarrow$
{\rm (iii)} \ Obvious.  

{\rm (iii)}  $\Rightarrow$ {\rm (ii)} \  If $q$ is a compact projection in $A^{**}$, then
by the discussion above Theorem \ref{peakch},  there exists $y_t \in {\rm Ball}(A)$
with $y_t \to q$, and $y_t q = q$.  Then $1-y_t \to 1-q$, and 
$(1-y_t)(1-q) = 1-y_t$, so $1-q$ is
open in  $(B^1)^{**}$, or equivalently is open in $(A^1)^{**}$ (by 
\cite[Theorem 2.4]{BHN}).   Hence $q$ is closed in $(A^1)^{**}$. 

(i) $\Leftrightarrow$ (ii) \  $q$ being 
 closed in $(A^1)^{**}$, or equivalently in  $(B^1)^{**}$ (by 
\cite[Theorem 2.4]{BHN}), 
is equivalent to $q$ being compact in $B^{**}$, by the $C^*$-algebra case.    
Note that this  implies that $q$ is closed  in $A^{**}$ since 
$e-q = e(1-q)$ is open in $(A^1)^{**}$ where  $e = 1_{A^{**}}$, being a
commuting product of open projections, hence
is open  in $A^{**}$.  

(ii)  $\Rightarrow$ (iv) \ 
Consider a projection $q \in A^{**}$ such that 
$q$ is closed in $(A^1)^{**}$.  Then $q^\perp$ is open in
$(A^1)^{**}$; let $C = q^\perp (A^1)^{**} q^\perp \cap A^1$, the HSA in $A^1$ with
support projection $q^\perp$.  Note that since $eq = q$ we have $(1-e) q^\perp = 1-e$,
and so $f = 1-e$ is a central minimal projection in $C^{**} = 
q^\perp (A^1)^{**} q^\perp$.    Let $D$ be the HSA in $A^1$ with support projection $e-q = e(1-q)$,
this is an approximately unital  ideal in $C$, indeed $D^{\perp \perp} = e  C^{**}$.  Note that $$C^{**}/D^{\perp \perp}
\cong C^{**} (q^\perp -  q^\perp e) =  C^{**} f = \Cdb f \cong \Cdb ,$$
The map implementing this isomorphism $C^{**}/D^{\perp \perp} \cong \Cdb f$ is the map
$x + D^{\perp \perp} \mapsto xf$.   Moreover, the map restricts to an
isometric  isomorphism $C/D \cong C/D^{\perp \perp} \cong \Cdb f$.  
This is because if the range of this restriction was not $\Cdb f$ then
it is $(0)$, so that $C = D$, which implies the contradiction $e = 1$. 
By \cite[Proposition 6.1]{BRead}, there is an element $d \in  {\mathfrak F}_C$ such that $df = 2f$.  If $b = d/2  \in  \frac{1}{2} {\mathfrak F}_C$, 
then $bf = f$.  We have that  $1-b \in \frac{1}{2} {\mathfrak F}_{A^1}$, and $(1-b) e = 1-b$, so 
$1-b \in \frac{1}{2} {\mathfrak F}_{A}$.  Moreover $(1-b)q = q$ since $b \in q^\perp (A^1)^{**} q^\perp$.   So $q$ is ${\mathfrak F}$-compact.
 \end{proof}  

From Theorem \ref{perun} (i) it is evident that compact projections in 
$A^{**}$ have many of the properties of Akemann's compact projections.
For example:

\begin{corollary} \label{ifcom} The infimum of any family of compact projections in 
$A^{**}$ is a compact projection in $A^{**}$. 
  Also, the supremum of two commuting compact projections 
 in
$A^{**}$ is a compact projection in $A^{**}$.
\end{corollary}

 \begin{proof}   Notice that these infima and suprema may be viewed as 
infima and suprema of projections in $A^{**}$ or in $B^{**}$, 
by remarks in the Introduction (particularly the third last paragraph before
Proposition \ref{uconv}).  We prove only 
the second statement, the first being similar.
This supremum may be viewed by Theorem \ref{perun}
as the supremum of two commuting closed
projections in $(B^1)^{**}$, which is closed by Akemann's theory,
and hence is compact in $A^{**}$ by Theorem \ref{perun} again.      
\end{proof}

\begin{corollary} \label{ifer}   Let  $A$ be an approximately unital
operator algebra, with approximately unital closed subalgebra $D$.
A projection $q$ in $D^{\perp \perp}$ is compact in $D^{**}$ iff
$q$ is compact in $A^{**}$.
\end{corollary}

\begin{proof}  One direction is obvious.
For the other, suppose that $q$ is compact in $A^{**}$.
Then $q$ is compact in $B^{**}$ by Theorem \ref{perun},
hence compact in $D^{**}$ by Theorem \ref{perun} again, since $D \subset B$.
  \end{proof}

The following is the analogue of \cite[Lemma 2.5]{AAP}.

\begin{corollary} \label{l25} Let  $A$ be an approximately unital
operator algebra.  If a projection $q$ in $A^{**}$ is dominated by an
open projection $p$ in $A^{**}$, then
$q$ is compact in $p A^{**} p$ (viewed as the
second dual of the HSA supported by $p$), iff $q$ is compact in $A^{**}$.
\end{corollary} 
 
\begin{theorem} \label{bncu}   {\rm (Refined noncommutative  Urysohn lemma
for operator algebras.) } \  Let  $A$ be an approximately unital
operator algebra.    Whenever a compact projection $q$ in $A^{**}$ is dominated by an
open projection $p$ in $A^{**}$, then there exists
$b \in \frac{1}{2} {\mathfrak F}_A$ with $q = qb, b = pb$.    Moreover, $q \leq u(b) \leq s(b) 
\leq p$.  
\end{theorem}

 \begin{proof}   If $q \leq p$ as stated,
then by Proposition \ref{l25} we know $q$ is compact in $D^{**} = pA^{**}p$, 
where $D$ is the HSA supported by $p$.  By  Theorem \ref{perun}  there 
exists $b \in  \frac{1}{2} {\mathfrak F}_D \subset \frac{1}{2} {\mathfrak F}_A$ 
with $q = qb, b = bp$.  Clearly $s(b) 
\leq p$.  We are unable to prove the other parts of the last statement yet, 
however they follow immediately 
from Corollary \ref{inc}.   
  \end{proof}

{\bf Remark.}  
%\begin{xrem}
  (1) \   In the case that $A$ is a $C^*$-algebra,
the above represents a
 proof of Akemann's noncommutative Urysohn lemma which
seems simpler than those
in the literature (we remark that in this
case the appeal to \cite[Proposition 6.1]{BRead} could be replaced by an
appeal to the fact that positive elements in a quotient $C^*$-algebra
lift to positive elements of the same norm).

In the setting of JB$^*$-triples it is shown in \cite{FP0}  that a 
tripotent is compact iff it is closed and `bounded' in an appropriate sense.

(2)  \ If $A$ is unital and  $q$ and $p$ are mutually orthogonal
compact projections in $A^{**}$, then there exists 
$a \in {\rm Ball}(A)$ with $aq = q$ and $a p = -p$, iff
 there exists $b \in  \frac{1}{2} {\mathfrak F}_A$ with $bq = q$ and $bp = 0$.
To see this simply use the formula $a = 2b-1$ or $b = \frac{a+1}{2}$.
%\end{xrem}

\medskip 

For interests sake, we give a
 different proof of our Urysohn lemma if $A$ is a uniform algebra.
 A (unital) uniform algebra is a closed unital subalgebra of $C(K)$
for compact  $K$.  An {\em approximately unital uniform algebra} is a 
Banach algebra $A$ with cai which is isometrically isomorphic to a 
subalgebra of a commutative $C^*$-algebra.  It is easy to see that this is 
the same as an ideal in a unital uniform algebra which has a
cai; one can take the unital uniform algebra to be the unitization $A^1$.  

\begin{proposition} \label{unu}  
If $A$ is an approximately unital
uniform algebra, and if  $q \in A^{**}$ is
a compact projection in $A^{**}$, and if
$u \in A^{**}$ is an  open
projection with $u \geq q$,
there exists $a \in 
\frac{1}{2} {\mathfrak F}_A$
 with
$a q = q$ and $a = au$.   If $A$ is
a unital
uniform algebra on a compact space $K$, then the above may
be restated in the language 
of $p$-sets: if $E$ and $F$ are disjoint $p$-sets
in $K$ for $A$, then there exists $a \in A$
such that $|1-2 a(z)| \leq 1$ for all $z \in K$,
which is $1$ on $E$ and $0$ on $F$.  
\end{proposition} 

  \begin{proof} 
First assume that $A$ is unital, acting on its maximal ideal space $M_A$.
To prove this case, note that the closed sets $E, F$ corresponding to $q$ and $u^\perp$
are disjoint $p$-sets, so $E \cup F$ is a $p$-set.   We follow
\cite[Proposition 4.1.14]{Dal}. By a simple Zorn's 
lemma argument, the ideal $J_E + J_F$, if this is not $A$,
 is contained in  a proper maximal ideal of $A$ (we recall that 
$J_E$ is the set of functions in $A$ vanishing on $E$).  This maximal ideal is
the kernel of a character, which corresponds to a point $w \in M_A$.
Since $f(w) = 0$ for all $f \in J_E$, we must have $w \in E$ (using
a well known property of $p$-sets).
Similarly, $w \in F$, so $w \in E \cap F = \emptyset$.
This contradiction shows that $J_E + J_F = A$.  Writing $1 = f + g$
with $f \in J_E, g \in J_F$, we have that $f = 1$ on $F$ and $f = 0$ on
$E$.   Let $g = 2f-1$.  By \cite[Theorem II.12.5]{Gamelin}, there exists $h \in {\rm Ball}(A)$ with
$h = g$ on $E \cup F$.   Let $a = \frac{1+h}{2}$, then $a \in 
\frac{1}{2} {\mathfrak F}_A$, and $a q = q$ and $a = au$.  

Now suppose that  $A$ is approximately unital.
If $q$  is compact in $A^{**}$ then as we said above,
$q$ is a closed projection  in $(A^1)^{**}$.
By the unital case there exists an $a \in 
\frac{1}{2} {\mathfrak F}_{A^1}$ with
$a q = q$ and $a = au$.   Since $a = ae$, where
$e$ is the support projection for $A$ in $(A^1)^{**}$, we have
$a \in A$, so $a \in 
\frac{1}{2} {\mathfrak F}_A$.
  \end{proof}

{\bf Remark.}
% \begin{xrem}
 (1) \ In contrast to the $C^*$-algebra case,
approximately unital operator algebras need not have 
any compact projections besides $0$.
Indeed if $A$ is an approximately unital
algebra of the type in \cite[Section 4]{BRead}, without
nontrivial open projections, then $A$ has no nontrivial compact projections.

(2) \ The Urysohn lemma for a  $C^*$-algebra $B$ may be sharpened
to the following, which we have not seen highlighted in the literature: 
Given projections $q \leq p$ in $B^{**}$, where $q$ is
 compact $q$ and $p$ open, then there exists $b \in B_+$,
and a compact $r$ with $q \leq b \leq s(b) \leq r \leq p$.  To prove
this, note that if we use the $b$ coming from
the usual Urysohn lemma  for $C^*$-algebras, 
then $u(b)$ and $s(b)$ may be regarded as
compact and open sets in $K$ where 
$C^*(b) = C_0(K)$.    By  the classical Urysohn lemma
there exists a nonnegative $g \in C_c(K)$ which is $1$ on the
compact set and $0$ on the complement of the open set.
Then $q \leq g(b) \leq s(g(b))   \leq r \leq s(b)  \leq p$ where 
$r$ corresponds to the compact support of $g$.

A similar proof gives an analogous Urysohn lemma for TRO's or JB$^*$-triples.
However the analogous result for nonselfadjoint operator algebras is
false.  Indeed this fails for the disk algebra, where closed projections
correspond to closed sets in the circle of measure zero, so that 
there cannot be open $p, u$ and closed $q, r$, 
with $0 \neq q \leq u \leq r \leq p \neq 1$.
%\end{xrem}   
 
\section{Compact projections and peak projections}

Let $B$ be a $C^*$-algebra.  Following \cite{Hay}, 
we call a projection $q$ in $B^{**}$  a {\em peak projection}
if there exists an $x \in {\rm Ball}(B)$
such that $xq=q$ and $\varphi(x^*x) < 1$ for all $\varphi \in Q(B)$ such that
     $\varphi(q) = 0$.   As in \cite{Hay} (see the proof of the 
next result),  this implies that $x^n \to q$ weak*.  This forces 
$q$ to be closed and hence compact (since with respect to 
the unitization, $1-x^n = (1-x^n)(1-q) \to 1-q$ weak*,
so $1-q$ is open, hence $e(1-q) = e-q$
is open, so that $q$ is closed).  We also say that $x$ peaks at
$q$, or $q$ is a peak for $x$, in this case.  
If $A$ is a subalgebra of $B$, and
$x \in A$ peaks at
$q$ in this sense, then since $x^n \to q$ weak* we
clearly have $q \in A^{**}$, and we say that $q$ {\em is a peak
projection for} $A$, or {\em is a peak
projection in} $A^{**}$.   If
$a$ peaks at a projection, then again from this fact about the 
limit of  $(x^n)$  it is easy to see that this projection 
is the largest projection such that $qa = q$.

\medskip

The following includes a version of \cite[Theorem 5.1]{Hay}  in our setting.  
In unpublished work  \cite{BH} with Hay from around 2006, referred to at the end of  
page 357 of \cite{BHN}, the first author
 proved a generalization of the following fact to TROs (see also \cite{LNW}).    
Here $u(x)$ is the tripotent mentioned in the 
introduction, namely  the weak* limit of $x (x^* x)^n$.

\begin{lemma} \label{haye}  Let $x \in {\rm Ball}(B)$ for a $C^*$-algebra $B$,
  and let $q$ be a closed projection in $B^{**}$ such that $xq=q$.  The following
  conditions are equivalent:
  \begin{enumerate}
 \item $x$ peaks at $q$,
   \item $\varphi (x^*x(1-q)) < 1$ for all $\varphi \in Q(B)$,
   \item $\varphi(x^*x(1-q)) < \varphi (1-q)$ for all $\varphi \in
     Q(B)$ such that $\varphi(1-q) \not= 0$,
   \item $\varphi(x^*x) < 1$ for every pure state $\varphi$ of $B$ such
     that $\varphi(q) = 0$,
   \item $\| px \| < 1$ for any compact projection $p$ in $B^{**}$
with $p \le 1-q$,
   \item $\| xp \| < 1$ for any compact projection $p \le 1-q$, and
   \item $\| xp \| < 1$ for any minimal projection $p \le 1-q$.
\end{enumerate}
These imply that   $q$ equals  the weak* limit of $(x^n)$, and $q$ 
also equals $u(x)$, the weak* limit of $x (x^* x)^n$.   Conversely, if this weak* limit $u(x)$ is a projection, then
this projection $u(x)$ is compact, indeed
$u(x) x = u(x)$, and the seven equivalent conditions above hold with $q = u(x)$.
\end{lemma}

\begin{proof}   Many parts of the proof
of \cite[Theorem 5.1]{Hay} carry through verbatim to the nonunital 
case of the seven numbered conditions, and the rest can be 
done by going to the unitization and applying  \cite[Theorem 5.1]{Hay} 
as we shall see, and we leave some of this to the reader.
For example, we demonstrate in the
next paragraph that 
if (4) holds then it also holds with $B$ replaced by $B^1$,
hence by \cite[Theorem 5.1]{Hay}, (1) and (2) for example hold with
$B$ replaced by $B^1$.  Since quasistates on $B$ have unique 
extensions to quasistates on $B^1$, it follows that 
(1) and (2) hold as stated.  
  
Thus suppose that (4) holds, and that $\varphi$ is a pure state on
$B^1$ with $\varphi(q) = 0$. If $\varphi(x^* x) = 1$, then $\psi =
\varphi_{\vert B} \in S(B)$.   If $\psi = \frac{1}{2}(\varphi_1 +
\varphi_2)$, with $\varphi_i \in Q(B)$, then $\varphi_i(x^* x) = 1$,
so that $\varphi_i$ are states.  Hence they have unique extensions
to states $\widetilde{\varphi_i}$ on $B^1$, whose average is clearly
$\varphi$.  Thus $\varphi = \widetilde{\varphi_i}$ and $\psi =
\varphi_i$.  Hence $\psi$ is pure, and so by hypothesis $\psi(x^* x)
= \varphi(x^* x) < 1$, a contradiction.  So (4) holds for $B^1$.

(1) $\Rightarrow$ (5) \
Assume (1) and let $p$ be a compact projection in $B^{**}$ such
that $p \le 1- q$.  Suppose $\|px\| = 1$.  Then $\|x^* px \| = 1$ and
$$x^* px  \le x^*(1- q)x = x^* x- q.$$
Since $p$ is closed, we have that $x^* px$ is a decreasing limit of
terms from $B$.  Thus  $x^* px$ is upper semicontinuous on
$Q(B)$. Thus, it achieves its maximum at a quasistate $\psi$,
which is necessarily a state. Hence,
$$1 = \psi(x^*px) \le \psi(x^* x-q) = \psi(x^* x)-\psi(q).$$  It
follows that $\psi(x^*x) = 1$ and $\psi(q) = 0$.  This contradicts (1).

That (1) implies that $q = u(x)$, and equals 
the weak* limit of $(x^n)$,  follows as in \cite[Lemma 3.6]{Hay}.  One considers the universal
representation of $B$, and  observes that
since $q x^* x q = q$ and
$(1- q) x^* x (1- q)$ is `completely nonunitary', 
we have $x^n \to q$ weak* and $(x^* x)^n \to q$ weak*.  It
follows that $x (x^* x)^n \to x q = q$ weak*, and so
$q = r = u(x)$.  In the last lines we have not used
that $q$ is closed, hence as we saw after the definition of 
a peak projection if $x q = q$, then (2) or (3), or the 
matching condition in the definition of a peak projection,
 implies that $q$ is closed.

Finally, if the weak* limit $u(x)$ of $x (x^* x)^n$  is a projection
$q$ say, then from e.g.\ \cite[Proposition 3.18]{BN0} and the lines 
above it, we have $x q = q$ and $q$ is a closed projection.
Hence $q x^* x = q$, so $q$ is a compact projection
with $q = qx$.   
Suppose that
$\varphi$ is a state of $B$ annihilating $q$.
In the universal representation of $B$ we can write
$\varphi = \langle \cdot \zeta , \zeta \rangle$ for a unit vector $\zeta \in H_u$.
If $\varphi(x^* x) = 1$, then $$\langle (1- q) x^* x (1- q) , \varphi
\rangle = 1,$$ since by Cauchy-Schwarz the other
terms in the expansion of the last displayed equation are $0$.
By the converse to Cauchy-Schwarz, $x^* x \zeta = \zeta$, and
$\varphi((x^* x)^n) = 1$.
Thus $\varphi(x^* q) = \varphi(q) = 1$, which is a contradiction.
So (1) in the lemma holds, hence also (2)--(7).
\end{proof}

\begin{corollary}  \label{isph}  If $a \in A$, and $q$ is  a 
projection in $A^{**}$, then $q$ is a peak for $a$ in $A^{**}$  in the sense above iff
 $q$ is a peak for $a$  in $(A^1)^{**}$   in the sense of {\rm \cite{Hay}}.  \end{corollary}  \begin{proof}  
This is obvious from the last proof, or is an exercise.  \end{proof} 

\begin{corollary}  \label{inc}  If $a \in \frac{1}{2} {\mathfrak F}_A$
then $u(a)$ is a projection, is in $A^{**}$, and is the peak for $a$.
Indeed  $q = u(a)$  satisfies all the 
equivalent conditions in the last lemma.   Also, $u(a) \leq s(a)$, and 
$u(a^n) = u(a^{\frac{1}{n}}) = u(a)$ if $n \in \Ndb$.  And $u(a) \neq 0$
iff $\Vert a \Vert = 1$.
\end{corollary}  \begin{proof}  
If $a \in \frac{1}{2} {\mathfrak F}_A$ for an approximately unital operator algebra $A$, then $\Vert 1 - 2a \Vert
 \leq 1$, and so $a = \frac{1}{2} (1 +y)$ where $y = 2a -1 \in {\rm Ball}(A^1)$.  By \cite[Proposition
6.7]{BHN}, if $p$ is the support projection of $\overline{(1-a)A^1}$, then $q = 1-p$ is a peak projection
for $a$.  So $a^n \to q$ and  $(a^*a)^n \to q$ weak* by \cite[Lemma 3.6]{Hay}.
So $a (a^*a)^n 
\to qa = q$ weak*.
Hence $q = u(a)$, and this is a projection.   Except for the last two statement,
the rest follows from Lemma
\ref{haye}.    

To see that $u(a) \leq s(a)$ recall  the power series representation  from \cite{BRead} , namely 
$a^{\frac{1}{n}} = \sum_{k = 0}^\infty \,
{t \choose k} (-1)^k (1-a)^k$, where $t = \frac{1}{n}$ .  Since $u(a) a = u(a)$ it follows that
 $u(a) a^{\frac{1}{n}} = u(a)$, and in the limit $u(a) s(a) = u(a)$. 

Finally, $u(a^{\frac{1}{m}})$ 
is the weak* limit of  $a^{\frac{nm}{m}})$, which is $u(a)$.  
Also, $a^m u(a) = u(a)$, and if $r$ is a minimal projection
dominated by $1 - u(a)$ then $\Vert r a^m \Vert \leq \Vert r a \Vert 
< 1$.  So by Lemma \ref{haye} we have $u(a^m) = u(a)$.
If $\Vert a \Vert < 1$ then clearly $(a^* a)^n \to 0$ and $u(a) = 0$.
If $\Vert a \Vert = 1$ then by a fact above Proposition \ref{uconv} 
there is a functional with  $\psi(u(a))  = 1$, so $u(a) \neq 0$.   \end{proof}

We now present our main result: 

\begin{theorem}  \label{inco}  
If $A$ is an  approximately unital operator algebra, then 
\begin{enumerate}   \item [{\rm (1)}]  A 
projection $q \in A^{**}$ is  compact iff it is  a decreasing 
limit of peak projections.  This is equivalent to $q$ being the infimum
of a set of peak projections.
\item [{\rm (2)}] If $A$  is a separable approximately unital operator algebra, then the compact
projections in $A^{**}$ are precisely the peak projections.
\item [{\rm (3)}]  A projection in $A^{**}$
is a peak projection in $A^{**}$   
iff it is of form $u(a)$ 
for some $a \in \frac{1}{2} {\mathfrak F}_{A}$. 
\end{enumerate}
\end{theorem}  \begin{proof}  (2) \ In the separable case, suppose that $q = q x$ for $x \in {\rm Ball}(A)$.  By \cite[Corollary 2.17 and Proposition 2.22]{BRead}, $1-q = s(w)$ for some $w \in
A^1$, and $q = u(z) = u(z) z$ for some $z \in {\rm Ball}(A^1)$.
Indeed $u(z)$ is a peak projection for $z$ in the sense of \cite{Hay}.
Let $a = zx \in {\rm Ball}(A)$.  Then $q a = qzx = qx = q$,
and for any compact projection $p \leq 1-q$ we have
$\Vert p a \Vert \leq \Vert p z \Vert < 1$ by \cite[Theorem 5.1]{Hay}.  So $q = u(a)$
by e.g.\ Lemma \ref{haye}. 

(1) \ One direction of the first `iff' is obvious.
For the other, let $q \in A^{**}$ be a compact
projection with $q = q x$ for some $x \in {\rm Ball}(A)$.  Then 
$q \leq u(x)$.
 Now  $1-q$ is an increasing limit of
 $s(x_t)$ for $x_t \in A^1$ with
$\Vert 1 - 2 x_t \Vert \leq 1$, by \cite[Corollary 2.21]{BRead}, 
so that $q$ is a decreasing weak* limit of the $q_t =
s(x_t)^{\perp} = u(1 - x_t)$
 (see \cite[Proposition 2.22]{BRead}). 
We have by Proposition \ref{uconv} that  $q_t \wedge u(x) = u(1 - x_t)
 \wedge u(x) = 
u(z)$, where $z = z_t$ is the average of $1 - x_t$
and $x$.
So $q_t \geq u(z)$.  However a tripotent dominated by a projection in the 
ordering of tripotents is a projection; thus $u(z)$ is a projection.   Also,
$q \leq u(z)$ since $q \leq q_t$ and $q \leq u(x)$.
Note that $(u(z_t))$ is decreasing,
since $(q_t)$ is decreasing.    
Let $a_t = z_t x  \in {\rm Ball}(A)$.  Then 
$$u(z_t) a_t = u(z_t) z_t x = u(z_t) x =  
u(z_t) u(z_t)^* u(x) u(x)^* x = u(z_t)  u(z_t)^* u(x) = u(z_t),$$ and 
for any compact projection $p \leq 1-u(z_t)$ 
we have
$\Vert p a_t \Vert \leq \Vert p z_t \Vert < 1$ by \cite[Theorem 5.1]{Hay}.  So $u(z_t)
= u(a_t)$ by e.g.\  Lemma \ref{haye}.
Then $u(a_t) = u(z_t) \searrow q$, since $q \leq  u(z_t)  \leq q_t \to q$ weak*.

Finally, the `infimum' statement follows as in \cite{Hay} from the fact that $u(a) \wedge u(b) = u(\frac{a+b}{2})$ (see  Proposition \ref{uconv}).

  (3) \ 
The one direction  is obvious.  For the other,
let $q = u(b)$ be a peak projection, for $b \in {\rm Ball}(A)$.  This is ${\mathfrak F}$-compact by Theorem \ref{perun}, so there exists $r
\in \frac{1}{2} {\mathfrak F}_A$ with $r q = q$.  If $d = r b \in {\rm Ball}(A)$
then $d q = rq = q$.  If $\varphi$ is a state on $B$ with $\varphi(q) = 0$,
then $\varphi(d^* d) \leq \varphi(b^* b) < 1$ by 
Lemma \ref{haye}.  By Lemma \ref{haye} again, $d$ peaks at $q$.
If $x = dr = rbr$ then a similar argument shows that 
$x \in {\rm Ball}(A)$ and $x$ peaks at $q$.   So $q = u(x)$.  Let $D$ 
be the separable operator algebra generated by $x$ and $r$.  Note that 
$s(r) \in D^{\perp \perp}$ and indeed $s(r)$ is an identity
for $D^{\perp \perp}$ (since $x = rbr$).  Thus $D$ is
an approximately unital operator algebra.   Let $G = C^*(D)$, the 
$C^*$-algebra generated by $D$ in $B$.  We now work in $D^1$ and $G^1$,
and its 
second dual.   The projection $f = 1-s(r)$ is a minimal projection in 
the center of $(G^1)^{**}$ (note that any cai for $D$ is a cai for $G$,
and hence the support projection $s(r)$ of $D$ is also the    
support projection of $G$ in $(G^1)^{**}$).  Hence $f$ is
closed.  Of course $q$ is a closed projection  in $(D^1)^{**}$, 
and $(1-s(r))q = q - s(r) q = 0$ since $q r = q$.  Therefore
$f + q$ is closed, hence compact.
By the second assertion in Theorem \ref{inco}, 
$f + q = u(k)$ for some $k \in {\rm Ball}(D)$.  
We have $k q = k u(k) q = u(k) q = q$.
  Let $h = 2r -1 \in {\rm Ball}(D^1)$, then $h q = 
2rq -q = q$, and $h f = 2rf -f = -f$.    Let $a = \frac{1}{2} (h k  +1) \in 
\frac{1}{2} {\mathfrak F}_{D^1}$.   
Note that 
$$a f = \frac{1}{2} (hk f + f) = \frac{1}{2} (hk u(k) f  + f) = \frac{1}{2} (hu(k) f  + f) =  \frac{1}{2} (hf  + f) =  \frac{1}{2} (-f+f),$$
so $af = 0$.   Hence $a \in D$.  We have $a q = \frac{1}{2} (hk q + q)
= \frac{1}{2} (h q + q) = q$.  Let $p$ be a minimal projection 
in $G^{**}$ with $p \leq 1-q$.   Then either $p f = 0$ 
or $p f = p$, or equivalently $p = ps(r)$ 
or $ps(r) = 0$.  Since $f$ is central, 
$\Vert a p \Vert$ equals
$$ \max \{ \Vert a p f \Vert , 
\Vert a p s(r) \Vert \} = \Vert a p \, s(r) \Vert  
 \leq \frac{1}{2} (\Vert hk p \, s(r) \Vert + 1) 
 \leq \frac{1}{2} (\Vert k p \, s(r) \Vert + 1) < 1,$$
since $\Vert k p \, s(r) \Vert < 1$ by Lemma \ref{haye}, because
 $p s(r)$ is $0$ or $p$, which 
is closed, and $$p s(r)  \leq (1-q) s(r) = s(r) - q = 1 - u(k) .$$ 
 Thus by      Lemma \ref{haye} again, 
$u(a) = q$ in $D^{**}$.  Clearly we also have 
$a \in  \frac{1}{2} {\mathfrak F}_{A}$, and $u(a) = q$ in $A^{**}$.  
\end{proof}

\begin{corollary}  \label{jmet} Let $a, b \in \frac{1}{2} {\mathfrak F}_{A}$
for an operator algebra $A$.
 If  $s(a)$ and $s(b)$ commute then their infimum   is
of form $s(c)$ for some $c \in \frac{1}{2} {\mathfrak F}_{A}$.  Similarly, if
$u(a)$ and $u(b)$ commute then 
their supremum is
of form $u(c)$ for some $c \in \frac{1}{2} {\mathfrak F}_{A}$.  
That is, 
the supremum of two commuting
peak projections in $A^{**}$ is a peak projection in $A^{**}$.
\end{corollary}  

\begin{proof}   Consider the separable operator 
algebra $D$ generated by $a$ and $b$.  Then $D^{**}$
 includes $s(a)$ and $s(b)$, and hence 
also includes
$e = s(a) \vee s(b)$.  Since $e$ is an identity for $D^{**}$,
$D$ is an approximately
unital operator algebra.   Clearly $s(a) \wedge s(b)$ is 
open in $D^{**}$, since
the infimum of two commuting open projections is
open (as proved by Akemann), and is 
 in $D^{**}$ by a remark in the `background and notation'
section of our introduction.  By \cite[Corollary 2.17]{BRead}, it
equals 
$s(c)$ for some $c \in \frac{1}{2} {\mathfrak F}_{D} \subset \frac{1}{2} {\mathfrak F}_A$.  

The second assertion is similar. 
Define $D$ as above, a separable approximately
unital operator algebra.  Then $u(a), u(b)$, and $u(a) \vee u(b)$,
are in $D^{**}$.   Clearly $u(a) \vee u(b)$ is closed 
(since the supremum of two commuting closed projections is closed),
and $$a (u(a) \vee u(b)) = a u(a) (u(a) \vee u(b)) = u(a) (u(a) \vee u(b)) =
u(a) \vee u(b). $$  So $u$ is compact.   By 
Theorem \ref{inco} (2),
$u(a) \vee u(b) = u(k)$ for some $k \in {\rm Ball}(D)$,
and by Theorem \ref{inco} (3) we may assume that $k \in \frac{1}{2} {\mathfrak F}_{D}
\subset \frac{1}{2} {\mathfrak F}_A$.  
The final assertion is clear from what comes before.
\end{proof}

We also point out a simple noncommutative variant of the well known
`Rossi local peak set theorem' from the theory of uniform 
algebras \cite{Gamelin}:  

\begin{corollary}  \label{rossi} 
Let
$A$ be an approximately
unital operator algebra. 
Suppose that $q$ is a closed projection in $A^{**}$
such that there exists an open projection $u \in A^{**}$ 
with $u \geq q$, and there exists $a \in {\rm Ball}(A)$ with
$a q = q$ and $\Vert a p \Vert < 1$ for every 
minimal projection $p \leq u - q$.
Then $q$ is a peak projection for $A$.
\end{corollary}

\begin{proof}  
By Lemma \ref{haye}, $q$ satisfies the conditions for being a peak projection in the second
dual of the HSA $D$ supported by $u$.    So there exists $a \in {\rm Ball}(D)$ with
$q = u(a)$.  Hence $q$ is also a peak projection for $A$.  
\end{proof}

% \begin{xrem}
{\bf Remark.} (1) \    The following illustrates a limitation of our theory.
A closed (even compact) projection in 
$A^{**}$ need not be the infimum of the open 
projections in
$A^{**}$ dominating it, in contrast to the $C^*$-algebra case (see \cite[Proposition 2.3]{Hay}).  Similarly,
an open  projection in 
$A^{**}$ need not be the supremum of the closed
projections in
$A^{**}$ which it dominates.
Indeed let $A$ be an approximately unital operator algebra with
no nontrivial open projections.  See e.g.\ \cite[Section 4]{BRead}.
As in that reference, $(A^1)^{**}$ has only one nontrivial open projection, namely 
the support projection $e$ of $A$ in $(A^1)^{**}$, and this is
clearly not  the supremum of the closed
projections in $A^{**}$ which it dominates.
And $q = 1 -e$ is  a closed (even compact) projection in $(A^1)^{**}$,
but is not  the infimum of the open
projections dominating it.

(2) We do not see a simple relationship between peak projections in the sense above,
for an approximately unital operator algebra,
and peak projections in the sense of {\rm \cite[Definition 2.20]{BRead}}.  Probably the
latter should not have been called peak projections.   In particular, one cannot say
for a peak projection $q$ in the sense of our present paper, that
$q^\perp = s(x)$
for some $x \in \frac{1}{2} {\mathfrak F}_A$ (although this is true if $A$ is separable
by  \cite[Corollary 2.17]{BRead}).    For example,
let $B = c_0(I)$ for  an uncountable discrete set $I$, and let $q = (0,1) \
\in B \oplus \Cdb$.
%\end{xrem}

\section{Compact projections and faces}
In this section we generalize some of the facial theory of $C^*$-algebras
from \cite{AP} to more general algebras.  In \cite{BHN} we began this; in Theorem 4.1 of that 
paper it was proved that, for a unital operator  algebra $A$,
a projection $p$ is open in $A^{**}$ if and only if $F_{p} = \{f \in S(A) : f(p)=0 \}$ is weak* closed.  It was pointed out that this remains true in the approximately unital case if one instead uses the quasi-state space $Q(A)$, which is weak* compact.  Hence, closed projections $q$ correspond to certain weak* closed faces $F_{1-q}$ of $Q(A)$.  
We will prove that compact projections correspond to certain faces $F_{1-q} \cap S(A)$ of $S(A)$ which are weak* closed
in $S(A)$ (or equivalently, in $Q(A)$). 

In this section and the next we assume that the cai for $A$ is a cai for the containing $C^*$-algebra $B$
(this is automatic if $A$ generates $B$).   For a norm one element $x \in A$, we let $\{ x \}^{\prime,S(A)}$ denote the weak*
closed face $\{f \in S(A) :  f(x)=1\}$ in $S(A)$, and we
let $\{ x \}^{\prime,A}$ denote the weak* closed face $\{f \in {\rm Ball}(A^{*}) :  f(x)=1\}$ in Ball$(A^{*})$.  Such
weak* closed faces are said to be {\em weak* exposed} (compare with \cite{ER}).
If $x \in A^{**}$, then $\{x\}_{\prime,S(A)}$ denotes $\{f \in S(A) :  f(x)=1\}$ and $\{x\}_{\prime,A}$ denotes $\{f \in {\rm Ball}(A^{*}) :  f(x)=1\}$.  Such norm closed faces are said to be {\it norm exposed}.  If $A$ is a $C^*$-algebra,
we write $\{ x \}^{\prime,A}$ and $\{x\}_{\prime,A}$ simply as $\{ x \}^{\prime}$ and $\{x\}_{\prime}$.
A face of $S(A)$ which is the  intersection of a family of
sets of the form $\{ a \}^{\prime,A}$ for $a \in {\rm Ball}(A)$,
is said to be {\em weak* semiexposed}.
It is clear that weak* semiexposed faces are weak* closed in $S(A)$.
If $q$ is a projection, it is clear that $F_{1-q} \cap S(A) = \{q\}_{\prime,S(A)}$.

The following is a very slight restatement of \cite[Lemma 3.3 (i)]{ER4}.

\begin{proposition} \label{exp}  Let $x  \in {\rm Ball}(A)$, and let $u(x)$
be the weak* limit of the sequence $(x (x^{*}x)^n)$ in $B^{**}$,
 for a containing $C^{*}$-algebra $B$. Then $\{u(x) \}_{\prime,B} =
 \{x\}^{\prime,B}$  and $\{ u(x) \}_{\prime,S(B)} = \{x\}^{\prime,S(B)}$.
If $u(x)$ is a projection (and hence coincides with the weak* limit of $(x^n)$ in
$A^{\perp\perp}$ by Lemma {\rm \ref{haye}}), then $\{
u(x) \}_{\prime,S(A)} = \{x\}^{\prime,S(A)}$.
 \end{proposition}
\begin{proof}
The first statement was proved in  \cite[Lemma 3.3 (i)]{ER4}, and 
discussed towards the end of our introduction.
 The second is immediate from the first; and the third from the 
second by considering Hahn-Banach extensions of states.  
\end{proof}

\begin{proposition}\label{proj}
If $A$ is an approximately unital operator algebra with $a \in {\rm Ball}(A)$.  Let $B$ be a containing $C^{*}$-algebra.  Then $u(a) \in B$ is a projection (and thus lies in $A^{\perp\perp}$) if and only if $\{a\}^{\prime,S(A)}=\{a\}^{\prime,A}$.
\end{proposition}
\begin{proof} Suppose that
$u(a)$ is a projection.  If $f \in \{a\}^{\prime,A}$, then $f(u(a)) = 1$
by Proposition \ref{exp}.  Considering the restriction of $f$ to
the two dimensional $C^*$-algebra Span$\{ u(a) , 1 \}$ it is clear that $f(1) = 1$,
so $f \in S(A)$.  Thus $\{a\}^{\prime,S(A)}=\{a\}^{\prime,A}$.

Now suppose that $\{a\}^{\prime,S(A)}=\{a\}^{\prime,A}$.
Then $\{a\}^{\prime,S(B)}=\{a\}^{\prime,B}$.
Since states are selfadjoint, it follows that
 $\{a\}^{\prime,B}=\{a^{*}\}^{\prime,B}$.  By
 Theorem 4.4 of \cite{ERF} we deduce that $u(a)=u(a^{*})=u(a)^{*}$.
So $u(a)$ is a selfadjoint partial isometry, and hence
by the spectral theorem equals $p-q$ for mutually orthogonal projections
$p, q \in B^{**}$.  If $q \neq 0$ there is a state $f$ on $B$ with
$f(q) = 1$.  Thus $0 \leq f(p) \leq f(1-q) = 0$.  So $f(p) = 0$
and $f(u(a)) = -f(q) = -1$, which by hypothesis forces $-f \in S(B)$.
Thus $f = 0$, which is a contradiction.   \end{proof}

Most of the following may be deduced from Corollary 4.4 of \cite{ER4} together with our Theorem 2.2, but we include a 
direct  proof.

\begin{proposition}\label{wso}
If $A$ is an approximately unital operator algebra and $q$ is a projection in $A^{**}$, then the following conditions are equivalent: 
\begin{enumerate} \item [{\rm (1)}] 
 $q$ is compact in $A^{**}$,
  \item [{\rm (2)}]  The face $\{q\}_{\prime, S(A)} = F_{1-q} \cap S(A)$ is weak* semiexposed,
\item [{\rm (3)}] The face $\{q\}_{\prime, S(A)} = F_{1-q} \cap S(A)$ is weak* closed in 
$S(A)$.
 \end{enumerate}
\end{proposition}
\begin{proof}  (1) $\Rightarrow$ (2) \  Suppose that $q$ is compact. By Theorem 
\ref{inco}, $q$ is a decreasing weak* limit of  projections $u(a_{\mu}) \in A^{**}$,
with $a_{\mu} \in {\rm Ball}(A)$. If $f \in S(A)$ and $f(u(a_{\mu})) = 1$ 
for all $\mu$, then in  the limit we have $f(q)=1$.
Conversely, if $f$ lies in $\{q\}_{\prime,S(A)}$, then 
$$1 = f(1) \ge f(u(a_{\mu})) \ge f(q) = 1. $$
 Hence $f$ lies in $\cap \{ a_{\mu} \}^{\prime,S(A)}$.  
Thus $\cap \{ a_{\mu} \}^{\prime,S(A)} =\cap \{ a_{\mu} \}^{\prime,A} = \{q\}_{\prime.S(A)}$, and so $\{q\}_{\prime,S(A)} $ is weak* semiexposed. 

(2)  $\Rightarrow$ (3) \ Obvious.

(3) $\Rightarrow$ (1) \ Considering Hahn-Banach extensions as in the proof of \cite[Theorem 4.1]{BHN}, it is clear that $\{q\}_{\prime,S(B)}$ is weak* closed in $S(B)$ for a containing
 $C^{*}$-algebra $B$. Hence $q$ is compact in $B^{**}$ by Lemma 2.4 of \cite{AAP}, and hence is compact in $A^{**}$
by Theorem \ref{perun}. 
\end{proof}  

{\bf Remark.}
% \begin{xrem}
 (1) \     As noted in \cite{BHN}, there may be
lots of weak* closed faces of $S(A)$ that are not of the form $\{ p \}_{\prime,S(A)}$.

(2)  \ Of course a compact projection $q$ equals 
$u(a)$ for some $a \in {\rm Ball}(A)$ iff $\{ q \}_{\prime,S(A)}$
is weak* exposed.  Also, this is equivalent to saying that 
the infimum of the $\{ u(a_{\mu}) \}$ appearing in the proof of 
(2) in the last result, equals 
$u(a)$ for some $a \in {\rm Ball}(A)$ (since as we discussed 
above, infima of projections correspond to intersections
of the matching faces \cite[Corollary 4.4]{ER4}).  
We also remark that Theorem \ref{inco} (3) is saying that 
if $b \in {\rm Ball}(A)$ and
$\{ b \}^{\prime,S(A)} = \{ b \}^{\prime,A}$,
 then $\{ b \}^{\prime,S(A)} = \{ a \}^{\prime,S(A)}$ for
some $a \in \frac{1}{2} \mathfrak{F}_A$.

(3) \ It is easy to see that the correspondence $\{ q \} \mapsto \{ q \}_{\prime,S(A)}$ is bijective
and order-preserving (by using the $C^*$-algebra case of this applied to Hahn-Banach extensions of states
of $A$).  
%\end{xrem}

\medskip

 Note that a state $\varphi$ of $A$ achieves its norm
at an element $a$ of $\frac{1}{2} \mathfrak{F}_A$ iff there exists a
compact projection $q \in A^{**}$ with $\varphi(q) = 1$.
To see the one direction of this set $q = u(a)$ and use
the fact above Proposition \ref{uconv}.  For the other direction,
if $q \leq u(a)$ for $a \in \frac{1}{2} \mathfrak{F}_A$, then
$1 = \varphi(q) \leq \varphi(u(a)) \leq 1$, so that $\varphi(u(a)) = \varphi(a)  = 1$.
Rephrasing this, we have:

\begin{corollary} \label{okc}  If $\varphi \in S(A)$, 
then $\varphi$ achieves its norm
at an element $a$ of $\frac{1}{2} \mathfrak{F}_A$ iff
 there exists   
a projection $q \in A^{**}$ with $\{ q \}_{\prime,S(A)}$ weak* 
closed and containing $\varphi$. 
\end{corollary}

Following \cite{AP} we have the following:

\begin{proposition} \label{2wfao}  
If $A$ is an approximately unital operator algebra and $p, q$ are projections in $A^{**}$, with 
$p$ open and $q$ compact, and $q \leq p$, then
$[p,q]_A = \{ a \in  \frac{1}{2} \mathfrak{F}_A : a q = q, ap = a \}$ is a  (nonempty)  norm closed
face of $\frac{1}{2} \mathfrak{F}_A$.
\end{proposition}
\begin{proof}   Clearly $[p,q]_A$ is  norm closed, and is nonempty by one of
 our Urysohn lemmas.   That $p$ is open and $q$ compact is only needed 
to get $[p,q]_A$ nonempty.   
Thus for the rest, we may assume that $A$ is unital and $p,q \in A$, by 
moving to  $A^{**}$.    In this case, if $x,y \in \frac{1}{2} \mathfrak{F}_A$ with $\frac{1}{2}(x+y) p = \frac{1}{2}(x+y)$,
then $\frac{1}{2}(p^\perp x p^\perp + p^\perp y p^\perp) = 0$.   Taking real parts,
$\frac{1}{2}(p^\perp (x + x^*) p^\perp + p^\perp  (y + y^*)  p^\perp) = 0$.
Since the real part of an element of  $\mathfrak{F}_A$ is positive, we deduce that
$p^\perp (x + x^*) p^\perp = 0$.     However if $x \in  \mathfrak{F}_A$ then 
it is easy to see that $x^* x  \leq x + x^*$, thus  $p^\perp  x^* x p^\perp \leq p^\perp (x + x^*) p^\perp = 0$.
Hence $x  p^\perp = 0$, so that $x p = p$.  Similarly $y p = p$.    By symmetry, replacing all elements by 
$1$ minus the element, we see that $(1-x) (1-q) = 1-x$, or $x q = q$.  Similarly $yq = q$.  
\end{proof}

Of course every $a \in  \frac{1}{2} \mathfrak{F}_A$ determines a face
$[u(a),s(a)]_A$ in which it lives.   The converse of the last result is false, 
namely that a norm closed
face of $\frac{1}{2} \mathfrak{F}_A$ need not equal 
$[q,p]_A$ for some  projections in $A^{**}$, with $p$ open and $q$ compact (in 
contrast to the situation for faces of the positive part of the unit ball in
a $C^*$-algebra \cite{AP}).  However it  might be  
interesting to  characterize such faces.

\medskip

{\bf Remark.}
%  \begin{xrem} 
See e.g.\ \cite{EF,FP2} for the characterizations 
of norm closed (resp.\ weak* closed) faces of the unit ball of a  JB$^*$-triple
(resp.\ its dual).  This is  much more difficult than the 
$C^*$-algebra case from \cite{AP}.  See \cite{ERF} for the (earlier) case of
weak* closed (resp.\ norm closed) faces of the unit ball of
JBW$^*$-triple (resp.\ its predual).  Also
see \cite{Neal}, for the earlier 
characterization of w*-closed faces of the quasi-state space
of a JB-algebra.
%\end{xrem} 

\section{Pure states and minimal projections}

In this section again $A$ is a closed subalgebra of a $C^*$-algebra $B$, and
we assume that $A$ has a  cai which is also a cai for $B$,
so that $1_{B^{**}} =
1_{A^{**}}$.  
We refer the reader to e.g.\ \cite{Ped} for the well known correspondences between pure states on a $C^*$-algebra $B$,
minimal projections in $B^{**}$, and maximal left ideals in $B$.   

\begin{proposition}  A minimal projection in $B^{**}$
which is also in $A^{\perp \perp}$, is
compact in $A^{**}$.
\end{proposition} \begin{proof}  This follows from
Theorem \ref{perun} (i) and the $C^*$-algebra case of the result.
 \end{proof}

{\bf Remark.}
%\begin{xrem}
   If $r$ is a minimal
projection
 in $B^{**}$ with $r \in A^{\perp \perp}$, then $(1-r) A^{**} \cap A$
is a maximal r-ideal (that is, right ideal with
left cai) of $A$ (and equals
$\{ a \in A : \varphi(a^* a) = 0 \}$ where $\varphi$
is the pure state on $B$ associated with $r$).   This is because any proper r-ideal
$J$ of $A$ containing $(1-r) A^{**} \cap A$ must satisfy
$(1-r) A^{**} \subset J^{\perp \perp} \subset A^{**}$.
The support projection of $J$ is $\geq e-r$, and hence must equal
$e-r$ since $r$ is minimal and $J$ is proper.  Thus $J = (1-r) A^{**} \cap A$.
The converse is false in general: maximal r-ideals
need not be associated with minimal
projections or with pure states.
%\end{xrem}

\medskip

As we said in the introduction,
states of $A$ are precisely the restrictions to
$A$ of states on $B$ (see \cite[2.1.19]{BLM}).
Also every extreme point of $S(A)$ is  the restriction to $A$ of
a pure state on $B$.  To see this,
if $\psi$ is an extreme point of $S(A)$, then the
set of Hahn-Banach extensions of $\psi$ is a convex weak*
closed subspace of $S(B)$, hence it has an extreme point,
which is easily seen to be an extreme point of $S(B)$.

In the following, oa$(a)$ is the operator algebra generated by $a$.  This
has a cai if $a \in  \frac{1}{2} \mathfrak{F}_A$ (see \cite{BRead}).

\begin{proposition} \label{iswha}  Let $a \in {\rm Ball}(A)$,
and suppose that  $u(a)$ is a projection.
Then $a$ achieves its norm at 
a pure state $\varphi$ of $B$, and at an 
extreme point of $S(A)$ (even at $\varphi_{\vert A}$). 
Also, there exists a minimal
projection
$r$ in $B^{**}$ such that $ar = r$.
If $B = C^*({\rm oa}(a))$, then there exists a unique 
pure state $\psi$ of $B$ with $\psi(a) = 1$, and this is a character (homomorphism).
\end{proposition} \begin{proof}   Basic functional
analysis tells us that there is 
a functional that achieves its norm at $a$.
By Proposition \ref{proj}, there is a state which 
achieves its norm at $a$.   The set $E$ of states taking value $1$ at $a$ is
a  weak* closed convex subset of $Q(A)$, so it has an extreme point 
$\varphi$ say.  If $\varphi = \frac{1}{2} (\varphi_1 +
\varphi_2)$ for states $\varphi_k$, then 
$1 = \varphi(a) = \frac{1}{2} (\varphi_1(a) +
\varphi_2(a))$, so that $\varphi_k \in E$
and so $\varphi_k = \varphi$.  Thus $a$
achieves its norm at an
extreme point of $S(A)$.  As we said above, $\varphi$ extends 
to a pure state on $B$.  If this pure state corresponds to a
minimal
projection
$r$ in $B^{**}$, then $r a r = r$, so that $ar = r$.
 
  Note that
$u(a)$ is a minimal, and a central,
 projection in $C^*({\rm oa}(a))^{**}$.
Thus if we define $\varphi_a$ by $\varphi_a(x) u(a) = x u(a)$,
for $x \in C^*({\rm oa}(a))^{**}$, then $\varphi_a$ is a character and
a pure state of $C^*({\rm oa}(a))$, and $\varphi_a(a) = 1$.
Since pure states extend, we may extend $\varphi_a$ to a
pure state of $B$.  
Conversely, suppose that $\psi$ is a pure state of $B$
such that $\psi(a) = 1$.   As we said above Proposition \ref{uconv},  
this means that
$\psi(u(a)) = 1$.  As above, there is a minimal projection 
$r$ in $B^{**}$ such that $\psi(x) r = rxr$ for all $x \in B$.
Thus $r u(a) r = r$, so that $r \leq u(a)$.  If $B = 
C^*({\rm oa}(a))$ it follows that $r = u(a)$, so that 
$\psi = \varphi_a$ on $C^*({\rm oa}(a))$.  Thus
$\varphi_a$ is the unique pure state on $C^*({\rm oa}(a))$ 
with value $1$ at $a$.  
\end{proof}   
 
{\bf Remark.}
% \begin{xrem}
 (1) \ Any  pure state on $C^*({\rm oa}(a))$          
different from $\varphi_a$
is also given by a minimal projection  $r \in C^*({\rm oa}(a))^{**}$,
but this time $r u(a) = 0$ (since $u(a)$ is central and not $r$), so that $r \leq 1 - u(a)$ and
so $\Vert a r \Vert < 1$ and $\Vert  r a \Vert < 1$ by 
Lemma \ref{haye}. 

(2)  \ If $a \in  \frac{1}{2} \mathfrak{F}_A$ has norm $1$ then
$u(a)$ is a projection, so the facts in the last proposition hold.
Moreover, the $\psi$ in the statement of that result restricts to 
an extreme point of $S({\rm oa}(a))$.   To see this, note that 
by the
first paragraph of the proof applied to $A = {\rm oa}(a)$,
there exists  $\rho \in {\rm ext}(S({\rm oa}(a)))$
with $\rho(a) = 1$.  Then $\rho$ extends to a pure state on $C^*({\rm oa}(a))$,
which must be $\varphi_a$.
% \end{xrem}

\medskip

If one considers the algebra $A$ of upper triangular $2 \times 2$ matrices with the
diagonal entries equal to each other,  it is clear that there exist two 
orthogonal minimal projections in $C^*(A) = M_2$ which restrict to the same
state on $A$.   From examples like this we see that pure states
of $C^*(A)$ need not `separate points' of $A$.  Also, 
it seems that a state of an 
operator algebra need not have a well defined `support projection' as exists in the
$C^*$-algebra theory.   Things seem to be better for  states that are extreme
points in $S(A)$ (perhaps the $q$ in the next result is similar to a support projection).  We note that $Q(A)$ has plenty of extreme points by the 
Krein-Milman theorem, and these are extreme points of $S(A)$.  Indeed this
argument shows that every state of $A$ is a weak* limit of convex combinations of 
extreme points of $S(A)$.   We now make some remarks about extreme points of $S(A)$.
 
First, if $\varphi \in {\rm ext}(S(A))$ then $\varphi$ need not achieve its norm
at an element of $\frac{1}{2} \mathfrak{F}_A$, unlike the $C^*$-algebra case.
Indeed if $A$ is a nonunital 
algebra of the type in \cite[Section 4]{BRead}, without
nontrivial open projections, then $A$ has no nontrivial compact projections.
If there existed $a \in \frac{1}{2} \mathfrak{F}_A$ with $\varphi(a) = 1$,
then $\varphi(u(a)) = 1$
as we said above Proposition \ref{uconv}, so that $u(a) \neq 0$.  Thus
$u(a) = e$, however $e$ is compact iff $e \in A$, which gives the 
contradiction that $A$ is unital.   In fact in such examples
 $\frac{1}{2} {\mathfrak F}_A$ contains no
norm $1$ elements besides $1$ (since $u(a)$  for such an element
is nonzero).   
 
Second, it seems unlikely that every $\varphi \in {\rm ext}(S(A))$
has a unique pure state extension to $B$. Note for example that the   
last part of the
last proof in \cite{Arv} should give rise to an explicit 
counterexample (and we thank Bill Arveson for a communication
on this point).  We also thank David Sherman for showing us
an example of a pure state on an operator system in $M_4$ with multiple
pure state extensions to the $C^*$-envelope.  
 However we have:

\begin{proposition}  \label{2dis} Let $\varphi \in {\rm ext}(S(A))$. 
 \begin{enumerate}  \item [{\rm (1)}]   
There exists a compact projection $q$ in $B^{**}$ such that
a pure state of $B$ extends $\varphi$ iff the
associated minimal projection is  dominated by $q$.
\item  [{\rm (2)}]   If $\varphi$ has
 more than one Hahn-Banach extension to $B$, then there
are two pure state extensions of $\varphi$ to $B$
 that are `mutually orthogonal',
that is, their associated minimal projections are mutually orthogonal.
\end{enumerate}   \end{proposition} 

 \begin{proof}  Consider the set $F$ of Hahn-Banach extensions
of $\varphi$ to $B$.   This is a face of $S(B)$
which is  weak* closed in $S(B)$.  Thus by the $C^*$-algebra
theory \cite{AP},  $F = \{ q \}_{\prime}$ for some 
compact projection in $B^{**}$.  Let $\psi$ be a pure state in $F$
(see the discussion at the start of this section),
with associated minimal projection $r$.  Since $r$ is
minimal, by the correspondence between 
projections and faces \cite{AP},
$\{ r \}_{\prime}$ is a minimal face.  It  is weak* closed
in $S(A)$ since $q$ is compact by Proposition
\ref{wso}.  It is also weak* closed in $Q(A)$, and so it has
extreme points.  By minimality, 
$\{ r \}_{\prime}$ is singleton, hence just contains  
$\psi$. Since
$\{ \psi \} \subset F$ we obtain $\{ r \}_{\prime}
\subset \{ q \}_{\prime}$, so that $r \leq q$ by the correspondence between 
compact projections and weak* closed faces \cite{AP}.  

Conversely, if $r$ is a minimal projection dominated by $q$,
then the associated pure state $\psi_r$  is in $\{ r \}_{\prime} \subset 
\{ q \}_{\prime} = F$.  So $\psi_r$ is an extension of $\varphi$.

 If $F$ is not a singleton, then since it is 
weak* closed in $Q(A)$, it is the weak* closed convex hull of its extreme points.
Hence $F$ has more than one extreme point.  These extreme points
are extreme points of $S(A)$, hence are exactly the 
pure states of $B$ which lie in $F$.
By the above, these correspond to minimal projections dominated by $q$.
Now $q B^{**} q$ is a von Neumann algebra with nontrivial atomic part.
In the atomic part there are two different minimal projections, and
hence by the structure of atomic von Neumann algebras there are 
two mutually orthogonal minimal projections.  Then consider
the associated pure states in $F$ as above.
\end{proof}

{\bf Closing Remark.}   One can ask if in Theorem \ref{peakch},  one may also choose $a$
 with $\Vert 1 - 2 a \Vert \leq 1$?   In this closing remark
we discuss
this issue.
 This is always true
iff for every compact projection $q \in A^{**}$, there exists a net $(y_t)$
in $\frac{1}{2} {\mathfrak F}_A$ such that $y_t q = q,$ and $y_t \to q$ weak*.
To see the one direction of this, substitute such $y_t$ into the proof
of Theorem \ref{peakch},
as in the proof of \cite[Theorem 2.24]{BRead}.  For the other direction,
proceed as in the proof of \cite[Corollary 2.25]{BRead}, but using
Theorem \ref{peakch} on the directed set of open $u \geq q$.
 One obtains $a_u \in \frac{1}{2} {\mathfrak F}_A$ such that $a_u q = q$,
and $(1-a_u)$ is a cai for the ideal in $A^1$ supported by $1-q$.
Thus $1- a_u \to 1-q$ and $a_u \to q$ weak*.

In this connection we remark that by the unital case,
there exists a net $(y_t)$
in $\frac{1}{2} {\mathfrak F}_{A^1}$ such that $y_t q = q,$ and $y_t  \to q$ weak*.
We also remark that by
\cite[Lemma 8.1]{BRead} there
 exists a net $(y_t)$
in $\frac{1}{2} {\mathfrak F}_A$ such that $y_t  \to q$ weak*.

We saw in Section 3
that there exists a net $(a_t)$
in $\frac{1}{2} {\mathfrak F}_A$ such that
 $u(a_t) \searrow q$.
This raises the question of whether
for $a \in \frac{1}{2} {\mathfrak F}_A$, if
we set $q = u(a)$, then does there
 exists a net $(y_t)$
in $\frac{1}{2} {\mathfrak F}_A$ such that $y_t q = q,$ and $y_t 
 \to q$ weak*? 
(Note that we cannot simply define the $y_t$ to be powers of $a$,
since these may leave $\frac{1}{2} {\mathfrak F}_A$.)   For
this latter question
one may assume that $A$ is the operator algebra oa$(a)$ generated by $a$, which
is commutative.
As we said in the last paragraph, there
 exists a net $(y_t)$
in $\frac{1}{2} {\mathfrak F}_A$ such that $y_t  \to q$ weak*.
Since $u(a)$ is a minimal and central  projection in $(C^*({\rm oa}(a))^{**}$,
we have $q y_t = \lambda_t q$, for scalars $\lambda_t$
which have limit $1$ (in fact it is easy to see also
that $|1- 2 \lambda_t| \leq 1$).

\subsection*{Acknowledgements}

We thank the referee for his welcome comments, and in particular 
suggesting some additional references.


\begin{thebibliography}{HD}

\normalsize \baselineskip=17pt

\bibitem{Ake} C. A. Akemann,  {\em The general Stone-Weierstrass problem,} J. Funct. Anal.\  4 (1969), 277--294.


\bibitem{Ake2} C. A. Akemann,  {\em Left ideal structure of
    $C^*$-algebras, } J. Funct. Anal.\ 6 (1970), 305--317.


\bibitem{AAP}   C. A.  Akemann, J. Anderson, and G. K. Pedersen, {\em Approaching infinity in $C^*$-algebras,} J. Operator Theory  21 (1989), 255--271.  

\bibitem{APc}   C. A.  Akemann and G. K. Pedersen, {\em Complications of semicontinuity in 
$C^*$-algebra theory,}  
Duke Math.\ J.\  40 (1973), 785--795. 

\bibitem{AP}   C. A.  Akemann and G. K. Pedersen, {\em Facial
structure in operator algebra theory,} Proc. London Math. Soc.\ 64  (1992), 418--448.


\bibitem{Arv}  W. B. Arveson, {\em The noncommutative Choquet boundary,} J. Amer. Math. Soc.\
  21 (2008), 1065--1084. 

\bibitem{BH}  D. P. Blecher and  D. M. Hay, {\em Peak tripotents,}  Unpublished
draft (2006).  

\bibitem{BHN}  D. P. Blecher, D. M. Hay, and
M. Neal, {\em Hereditary subalgebras of operator algebras,} J.\
Operator Theory 59 (2008), 333-357.

\bibitem{BLM}  D. P. Blecher
and C.  Le Merdy, {\em Operator algebras and their modules---an
operator space approach,} Oxford Univ.\  Press, Oxford (2004).

\bibitem{BN0}   D. P. Blecher and M. Neal, {\em Open partial isometries and
positivity in operator spaces,} Studia  Math.\    182 (2007), 227-262.

\bibitem{BN}  D. P. Blecher and M. Neal, {\em Open projections in 
operator algebras I: Comparison theory,} To appear, Studia Math.   
 

\bibitem{BRead}   D. P. Blecher and C. J. Read, {\em Operator algebras
with contractive approximate identities},  J.\
Funct.\ Anal.\  261 (2011), 188-217.

\bibitem{Brown}  L. G. Brown, {\em Semicontinuity and multipliers of $C^*$-algebras,}
Canad.\ J.\ Math.\  40 (1988),  865-–988.

\bibitem{Dal}  H. G. Dales,
{\em Banach algebras and automatic continuity},
London Mathematical Society Monographs.
New Series, 24, Oxford Science Publications.
The Clarendon Press, Oxford University Press, New York, 2000.



\bibitem{EF}  C. M. Edwards, F. J. Fernandez-Polo, C. S. Hoskin,
and A. M. Peralta, {\em  On the facial structure of the unit ball in a JB$^{*}$-triple,}
 J. Reine. Angew.\  Math.\  64 (2010), 123--144.


\bibitem{ERF}  C. M. Edwards and G. T. R\"uttimann, {\em  On the facial structure of the unit balls in a JBW$^{*}$-triple and its predual,}
 J. London. Math. Soc.\   641 (1988), 317-332.


\bibitem{ER4}  C. M. Edwards and G. T. R\"uttimann, {\em Compact tripotents in
bi-dual JB$^*$-triples,} Math. Proc. Camb. Philos. Soc.\ 120  (1996), 155-173.

\bibitem{ER} C. M. Edwards and G. T. R\"uttimann,  {\em
Exposed faces of the unit ball in a JBW$^*$-triple},
Math. Scand. 82 (1998), 287--304.

\bibitem{FP0}  F. J. Fernandez-Polo and A. M. Peralta, {\em
Closed tripotents and weak compactness in the dual space of a JB$^*$-triple,}
 J.\ London Math.\ Soc.\   74 (2006),  75--92.  

\bibitem{FP1}   F. J. Fernandez-Polo and A. M. Peralta, {\em
Compact tripotents and the Stone-Weierstrass theorem for $C^*$-algebras and JB$^*$-triples,}
 J.\ Operator Theory  58 (2007),  157--173.

\bibitem{FP}  F. J. Fernandez-Polo and A. M. Peralta, {\em
Non-commutative generalisations of Urysohn's lemma
  and hereditary inner ideals,}  J. Funct.\ Anal.\   259 (2010), 343--358. 

\bibitem{FP2}   F. J. Fernandez-Polo and A. M. Peralta, {\em On the facial structure of the unit ball of
the dual space of a JB$^{*}$-triple,}  Math. Ann.\  348
(2010), 1019--1032.

\bibitem{Gamelin}   T. W. Gamelin, {\em Uniform Algebras,} Second edition,
Chelsea, New York, 1984.

\bibitem{Hay}  D. M. Hay, {\em Closed projections and peak interpolation for operator algebras,}
  Integral Equations Operator Theory   57  (2007),  491--512.

\bibitem{Hayr}  D. M. Hay, {\em Multipliers and hereditary subalgebras of 
operator algebras,}  Studia Math.\  205 (2011), 31-40.  

\bibitem{LNW}  C.W. Leung, C.K. Ng, and N.C. Wong,  {\em
Geometric pre-ordering on $C^*$-algebras,}  J. Operator Theory 
 63  (2010),  115--128.

\bibitem{Neal}  M. Neal, {\em Inner ideals and facial structure of the quasi-state space of a 
JB-algebra,} J.\ Funct.\ Anal.\  173 (2000), 284--307.

\bibitem{ORT} E. Ortega, M. R{\o}rdam, and H. Thiel, {\em The Cuntz semigroup and
comparison of open projections,}   J.\ Funct.\ Anal.\  260 (2011),
3474--3493.




\bibitem{Ped} G. K. Pedersen, {\em $C^*$-algebras and their automorphism
groups,} Academic Press, London (1979).

\bibitem{Read}  C. J. Read, {\em On the quest for positivity in operator algebras,}
 J. Math. Analysis and Applns.\  381 (2011),  202--214.

\end{thebibliography}
\end{document}